\documentclass[12pt,amscd]{amsart}
\usepackage[all]{xy}
\usepackage{graphicx}
\usepackage{amsmath,amsxtra,amssymb,latexsym, amscd,amsthm}
\usepackage{indentfirst}
\usepackage[mathscr]{eucal}

\newtheorem{thm}{Theorem}[section]
\newtheorem{cor}[thm]{Corollary}
\newtheorem{lem}[thm]{Lemma}

\theoremstyle{definition}

\newtheorem{exm}[thm]{Example}
\newtheorem*{rem}{Remark}

\DeclareMathOperator{\NN}{\mathbb {N}}
\DeclareMathOperator{\ZZ}{\mathbb {Z}}

\DeclareMathOperator{\lk}{lk}
\DeclareMathOperator{\st}{st}

\DeclareMathOperator{\reg}{reg}

\DeclareMathOperator{\core}{core}

\def\B {\mathcal B}
\def\C {\mathcal C}
\def\S {\mathcal S}
\def\I {\mathcal I}

\def\a {\mathbf a}
\def\b {\mathbf b}
\def\x {\mathbf x}
\def\0 {\mathbf 0}
\def\m {\mathfrak m}
\def\F {\mathcal F}
\def\D {\Delta}

\def\h {\widetilde{H}}

\begin{document}

\title[Regularity of symbolic powers and Arboricity of matroids] {Regularity of symbolic powers and Arboricity of matroids}

\author{Nguy\^en C\^ong Minh}
\address{Department of Mathematics, Hanoi National University of Education, 136 Xuan Thuy, Hanoi,
Vietnam}
\email{minhnc@hnue.edu.vn}
\author{ Tr\^an Nam Trung}
\address{Institute of Mathematics, VAST, 18 Hoang Quoc Viet, Hanoi, Vietnam}
\email{tntrung@math.ac.vn}

\date{}

\dedicatory{}
\commby{}
%-----------------------------------------------------------
\begin{abstract} Let $\Delta$ be a simplicial complex of a matroid $M$. In this paper, we explicitly compute the regularity of all the symbolic powers of a Stanley-Reisner ideal $I_\Delta$ in terms of combinatorial data of the matroid $M$. In order to do that, we provide a sharp bound between  the arboricity of $M$ and the circumference of its dual $M^*$. 

\end{abstract}
%In particular, this function is a linear function in $t$ for all $t\geqslant 1$.
%$\reg(I_M^{(t)}) = c(M)(t-1)+(r(\core(M))+1)$ for all $t\geqslant 1$$I_M$ be the Stanley-Reisner ideal of $M$. Then .
% -----------------------------------------------------------
\maketitle
% -----------------------------------------------------------

\section{Introduction}

Let $S = K[x_1,\ldots, x_n]$ be a polynomial over a field $K$ and $I$ a homogeneous ideal in $S$. The Castelnuovo-Mumford regularity (or regularity) of $I$, $\reg (I)$, is a kind of universal bound for important invariants, such as the maximum degree of the syzygies and the maximum non-vanishing degree of the local cohomology modules of $I$.  

It is well-known that (see \cite{CHT, K,TW}) there exist nonnegative integers $d, e$ and $t_0$ such that $\reg(I^t)=dt+e$ for all $t\geqslant t_0$. While $d$ is determined by $I$ (see \cite{K}), a little is known about $e$ and $t_0$. From this remarkable result, Eisenbud and Ulrich \cite{EU} posed the following problems: What is the significance of the number $e$?  What is a reasonable bound $t_0$? These problems continue to attract us (see e.g. \cite{Ber, BHT, C,  EH, Ha, HTT, HHTT}). 

Moving away from ordinary powers, for symbolic powers $I^{(t)}$, it was shown in small dimensional cases that the function $\reg(I^{(t)})$ is bounded by a linear function in $t$ (see \cite{HHT, HT1}). However, Cutkosky \cite{Cu} constructed a non-monomial ideal $I$ such that $\lim_{t\to \infty}\dfrac{\reg(I^{(t)})}{t}$ is not rational, so that the asymptotic behavior of  $\reg(I^{(t)})$ is in general far from being a linear function. For the case $I$ is a monomial ideal, in \cite{HHiT}, they proved that this function is quasi-linear for large $t$. Moreover, if $I$ is a squarefree monomial ideal, then $\lim_{t\to \infty}\dfrac{\reg(I^{(t)})}{t}$ exists due to \cite[Theorem 4.9]{HT1}. Up until now, in that case it is not known whether $\reg(I^{(t)})$ is a linear function in $t$ for $t\gg 0$ or not. If one can add more that $I$ has codimension $2$, then this function will be linear in $t$ for $t\geqslant 2$ by \cite{HT2}. Thus, it is natural to ask the following question:

\medskip

\noindent {\bf Question.} {\it Let $I$ be a (squarefree) monomial ideal in the polynomial ring $S$. Is $\reg(I^{(t)})$ a linear function in $t$ for large $t$?}

\medskip

In this paper, we explicitly compute $\reg(I_{\Delta}^{(t)})$ for every positive integer $t$ where $\Delta$ is a simplicial complex of a matroid $M$ and $I_\D$ the Stanley-Reisner.  In particular, it is a linear function in $t$ for all $t\geqslant 1$, so the question is settled affirmatively in this case. For a matroid $M$, we denote by $r(M)$, $c(M)$ and $\core(M)$ to be the rank, the circumference and the core of $M$, respectively. Then, the main result of the paper is the following.

\medskip
\noindent {\bf Theorem \ref{regsym}.} {\it Let $\D$ be the simplicial complex of a matroid $M$.  Then, $$\reg(I_{\Delta}^{(t)}) = c(M)(t-1)+r(\core(M))+1, \ \text{ for all } t\geqslant 1.$$
}

As an application we can characterize any matroid whose the $t$-th symbolic power of the Stanley-Reisner ideal has linear resolution (see Theorem $\ref{regsym-cor}$). 

An important tool used in  recent researches on the symbolic powers of a squarefree monomial ideal is a Hochster's type formula of computing local cohomology modules in terms of reduced (co)homology groups of some certain simplicial complexes (see \cite{HKTT, HT1, HT2, HMT, MT,  MTT, TT}). In order to study the above question it is  possible that additional hypotheses on $I^{(t)}$ may be required; starting with the case where $I^{(t)}$ is Cohen-Macaulay is probably a good idea. In \cite{TT}, they proved that $I^{(t)}$ is Cohen-Macaulay for some $t\geqslant 3$ (and then for all $t\geqslant 1$) if and only if $I$ is a Stanley-Reisner ideal of a matroid complex.

On the other hand,  matroids provide a link between graph theory, linear algebra, transcendence theory, and semimodular lattices (see \cite{O, W}). Let $M$ be a matroid on the ground set $V$ and let $\B(M)$ be the set of all bases of $M$. The arboricity of $M$, denoted by $a(M)$, is the minimum number of bases needed to cover $V$.  A well-known result of Edmonds \cite{Ed}, extending the result of  Nash-Williams \cite{NW} for graphic matroids, says that
$$a(M)=\max\left\{\left\lceil\dfrac{|A|}{r(A)}\right\rceil \mid A\subseteq V\right\};$$
where $r(A) = \max\{|A\cap B| \mid B \in \B(M)\}$.

In order to get an explicit formula for $a(M)$ is difficult, and so the estimate of this invariant seems to be of independent interest (see e.g. \cite{AMR, DHS, Se}). This is motivated from some computationally intractable problems on a general database.  The second main result of the paper establishes a bound for $a(M)$ in terms of the circumference of its dual matroid $M^*$. It also plays a key role in the proof of Theorem $\ref{regsym}$.

\medskip

\noindent{\bf Theorem \ref{Arbor}.} {\it Let $M$ be a matroid. Then, $a(M) \leqslant c(M^*)$. 
}

\medskip

\indent Now we explain the organization of the paper. In Section 2, we recall some notations and basic facts about  the Stanley-Reisner ideal, matroids, and Castelnuovo-Mumford regularity. Section 3 contains the proof of Theorem \ref{Arbor}. And, in the last section we devote to the proof of Theorem \ref{regsym} and applications.    

\section{Preliminaries}

In this section, we recall some definitions and properties concerning simplicial complex, matroids, degree complexes that will be used later. The interested reader is referred to (\cite{Ei,O,S}) for more details.

\subsection{Simplicial complex} 
Let $\Delta$ be a simplicial complex on $[n]=\{1,\ldots, n\}$ that is a collection of subsets of $[n]$ closed under taking subsets. We put $\dim F = |F|-1$, where $|F|$ is the cardinality of $F$, and
$\dim \Delta = \max \{ \dim F \mid F \in \Delta \}$, which is called the
dimension of $\Delta$.  It is clear that $\Delta$ can be uniquely  determinate by the set of its maximal elements under inclusion, are called by facets, which is denoted $\F(\Delta)$. The complex $\D$ is said that pure if all its facets have the same cardinality.  

Let $K$ be a field and $S=K[x_1,\ldots, x_n]$ a polynomial ring over $K$. The Stanley-Reisner ideal $I_\Delta$ of $\Delta$ (over $K$) is the ideal in $S$ which generated by all square-free monomials $x_{i_1}\ldots x_{i_p}$ such that $\{i_1,\ldots, i_p\}\notin\Delta$. The quotient ring $S/I_\Delta$ is called the Stanley-Reisner ring of $\Delta$. We say that $\D$ is Cohen-Macaulay (over $K$) if so is the Stanley-Reisner ring $S/I_\D$. For a face $F\in\Delta$ and a subset $S\subseteq [n]$ we define the link of $F$ in $\Delta$ to be
$$\lk_{\Delta}F=\{G\in\Delta \mid  F\cup G\in\Delta, F\cap G=\emptyset\}$$
and the restriction of $\Delta$ to $S$ to be
$$\Delta[S] =\{F\in\Delta \mid F\subseteq S\}.$$ 
If $S=[n]\setminus\{u\}$ for some $u\in [n]$, then we will write $\D_{-u}$ instead of $\D[S]$ for simplicity.

\subsection{Matroids}

A matroid $M$ on the ground set $V$ is a collection $\I$ of subsets of $V$, which is called independent sets, satisfying the following conditions:
\begin{enumerate} 
\item[(i)] $\emptyset\in\I$,
\item[(ii)] If $I\in\I$ and $J\subseteq I$, then $J\in\I$,
\item[(iii)] If $I, J\in\I$ and $|J|<|I|$, then there exists an element $x\in I\setminus J$ such that $J\cup\{x\}\in \I$.
\end{enumerate}

Maximal independent sets of $M$ are called bases. They have the same cardinality which is called the rank of $M$, and denoted by $r(M)$. Let $\B(M)$ be  the set of all bases of $M$. A dependent set is a subset of $V$ which is not in $\I$. Minimal dependent sets are called circuits of $M$. Denote by $\C(M)$ the set of all circuits of $M$. It is clear that $\C(M)$ determines $M$: $\I$ consists of subsets of $V$ that do not contain any member of $\C(M)$. The circumference of a matroid $M$ is defined by 
$$c(M) =\max\{|C| \mid C\in \C(M)\}.$$
The matroid $M$ is a star with a center $x\in V$ if $x\in B$ for any $B\in\B(M)$. 

\begin{exm} For a simple graph $G$ we define the graphic matroid of $G$, denoted by $M(G)$, to be a matroid whose independent sets are the forests in $G$. Then, the bases of the graphic matroid $M(G)$ are the spanning forests of $G$, and the circuits of $M(G)$ are the simple cycles of $G$; where a simple cycle means a cycle without chords. 

The cographic matroid of $G$ is just the dual of $M(G)$ and denoted by $M^*(G)$. Recall that a cut-set is a collection of edges $S$ of $G$, such that when the edges in $S$ are deleted from $G$, the number of connected components of $G$ increases by one; and a bond is a cut-set of $G$ that does not have any other cut-set as a proper subset. Accordingly,  circuits in $M^*(G)$ are the bonds of $G$. 
\end{exm}

Let $A$ a subset of $V$. Let $\I|A=\{I\subseteq A \mid I\in\I\}$. Then, $\I|A$ is a also matroid over $A$, denoted by $M|A$. If $U=V\setminus\{u\}$, for $u\in V$, we will write $M_{-u}=M|U$ for short . We define the rank $r(A)$ of $A$ to be the size of a basis $B$ of $M|A$ and a such set is called a basis of $A$. One easily check that $$\C(M|A)=\{C\subseteq A \mid  C\in\C(M)\}.$$ 
Let $\B^*(M)=\{V\setminus B\mid B\in\B(M)\}$. It is known that the set $\B^*(M)$ forms the set of bases of a matroid on $V$ which is called the dual of $M$, is denoted by $M^*$. And, it is clear $(M^*)^*=M$. 

\begin{exm} The uniform matroid $U_{k,n}$ is defined over the ground set $[n]$. A subset of $[n]$ is independent if and only if it contains at most $k$ elements. Thus, a subset is a basis if it has exactly $k$ elements, and it is a circuit if it has exactly $k+1$ elements. Its dual $(U_{k,n})^*$  is also another uniform matroid $U_{n-k,n}$.
\end{exm}

It is apparent from the definition that the collection of independent sets of a matroid $M$ forms a simplicial complex, which is called  matroid complex (or independence complex) of $M$. This one is a pure simplicial complex of dimension $r(M)-1$. 

\begin{lem} \label{star} Let $M$ is a matroid such that it is not a star. Let $\D$ be the matroid complex of $M$  and $F\in\D$. If  $\lk_\D(F)\ne\emptyset$, then $\lk_\D(F)$ is a matroid complex that is not a cone.
\end{lem}
\begin{proof} It suffices to prove the lemma in the case $F=\{x\}$ for $x\in V$. Note that $\lk_\D(x)$ is also a matroid by definition. Assume on the contrary that $\lk_\D(x)\ne\emptyset$ is a cone for some $x\in V$. Let $y$ be a center of this cone. Obviously, $y\ne x$. Since $M$ is not a star, there exists $B\in \B(M)$ such that $y\notin B$ (i.e. $x\notin B$). Put $F\in\F(\lk_\D(x))$, then $F\cup \{x\}\in\B(M), x\notin F$. Therefore, $F'=(F\cup \{x\})\setminus\{y\}\in M$ and $|(F\cup \{x\})\setminus\{y\}|<|B|$. By the definition of matroids, there exists $z\in B\setminus F'$ such that $F'\cup\{z\}\in \B(M)$. Thus, $(F'\cup\{z\})\setminus\{x\}\in\F(\lk_\D(x))$ and $y\notin (F'\cup\{z\})\setminus\{x\}$, which is a contradiction, and the lemma follows.
\end{proof}

We will also need the following property of a matroid (see \cite[Theorem 3.4]{S}).

\begin{lem} \label{acyclic} If $\D$ be a matroid complex, then $\D$ is a cone if and only if it is acyclic (i.e., has vanishing reduced homology).
\end{lem}

\subsection{Castelnuovo-Mumford regularity, Symbolic power and degree complexes}

Let $\m = (x_1,\ldots, x_n)$ be the maximal homogeneous ideal of $S$. For a finitely generated graded $S$-module $L$, let
$$a_i(L)=\sup\{j\in\ZZ \mid H_{\m}^i(L)_j \ne 0\}$$ 
where $H^{i}_{\m}(L)$ denotes the $i$-th local cohomology module of $L$ with respect $\m$. Then, the Castelnuovo-Mumford regularity (or regularity for short) of $L$, denoted by $\reg(L)$, is defined by
$$\reg(L) = \max\{a_i(L) +i\mid i = 0,\ldots, \dim L\}.$$

The regularity of $L$ also defines via the graded minimal free resolution. Assume that the graded minimal free resolution of $L$ is
$$0\longleftarrow L\longleftarrow F_0\longleftarrow F_1\longleftarrow\cdots\longleftarrow F_p\longleftarrow 0.$$
Let $t_i(L)$ be the maximal degree of graded generators of $F_i$. Then,
$$\reg(L) = \max\{t_i(L) - i\mid i = 0, \ldots, p\}.$$

Let $J$ be a non-zero and proper homogeneous ideal of $S$. From the graded minimal free resolution of $S/J$ we obtain $\reg(J)=\reg(S/J)+1$. We say that $J$ has a linear resolution if all entries in the matrices representing the differentials in a graded minimal free resolution of $J$ are linear forms. 

Let $\{P_1,\ldots,P_r\}$ be the set of the minimal prime ideals of $J$. Given an positive integer $t$, the $t$-th symbolic power of $J$ is defined by
$$J^{(t)}=\bigcap_{i=1}^r J^tS_{P_i}\cap S.$$
In particular, if $J=I_\D$ is the Stanley-Reisner ideal of a simplicial complex $\D$, putting $P_F=(x_i\mid i\notin F)S$ for each facet $F\in\F(\D)$, then $I_\D=\bigcap_{F\in\F(\D)}P_F$, and then
$$I_\D^{(t)}=\bigcap_{F\in\F(\D)}P_F^t.$$

Let $\D$ be the matroid complex of a matroid $M$ on the ground set $[n]$. Let $I$ be the Stanley-Reisner ideal of $\D$ in $S$. One can see that $\B(M)=\F(\D)$ and $$I=(x^C \mid C\in\C(M))=\bigcap_{B\in\B(M)}P_B,$$
where $x^C=\prod_{i\in C}x_i$, and hence  $I^{(t)}=\bigcap_{B\in\B(M)}P_B^t$.  

Let $$\core([n])=\{i\in[n]\mid \st_{\D}(i)\ne \D\},$$ 
where $\st_\D(i)=\{F\in\D\mid F\cup\{i\}\in\D\}$, and $\core(\D)=\D[\core([n])]$ (also write w.r.t. $\core(M)=M|\core([n])$).
It is clear that $\D[[n]\setminus\core([n])]$ is a simplex and $\{x_i\mid i\in[n]\setminus\core([n])\}$ forms a linear regular sequence of $S/I^{(t)}$. Therefore,  $$\reg(I^{(t)})=\reg(I_{\core(\D)}^{(t)}).$$

For simplicity of exposition, throughout the rest of this paper, we can assume $M=\core(M)$, i.e. $M$ is not a star and in this case, $\D=\core(\D)$ and $n\geqslant 2$. 

\medskip

Let $I$ be a monomial ideal in $S$. Takayama in \cite{Ta} found a combinatorial formula for $\dim_KH_\m^i(S/I)_\a$ for all $\a\in\ZZ^n$ in terms of certain simplicial complexes. For every $\a = (a_1,\ldots, a_n) \in \ZZ^n$ we set $G_\a = \{i\mid \ a_i < 0\}$ and write $\x^{\a} = \Pi_{j=1}^n x_j^{a_j}$. Thus, $G_\a =\emptyset$ whenever $\a \in \NN^n$. The degree complex $\D_\a(I)$ is the simplicial complex whose faces are sets of form 
 $F \setminus G_\a$, where $G_\a\subseteq F\subseteq [n]$, so that for every minimal generator $x^\b$ of $I$ there exists an index $i \not\in F$ with $a_i < b_i$. To present $\D_\a(I)$ in a more compact way, for every subset $F$ of $[n]$ let $S_F :=R[x_i^{-1} \mid i\in F \cup G_\a]$. Then, by \cite[Lemma 1.2]{MT} we have
$$\D_\a(I) = \{F\subseteq [n]\setminus G_\a\mid \x^\a \notin IS_F\}.$$

Let $\D(I)$ denote the simplicial complex that corresponds to the square-free monomial ideal $\sqrt{I}$. Note that $\Delta(I_{\Delta}^{(t)}) = \D$. We have:

\begin{lem}[\cite{Ta}, Theorem 1] \label{T}
$$\dim_KH_\m^i(S/I)_\a = 
\begin{cases}
\dim_K\widetilde H_{i-|G_\a|-1}(\D_\a(I);K) & \text{\rm if }\ G_\a \in \D(I)\ ,\\
0 & \text{\rm otherwise. }
\end{cases} $$
\end{lem}

The next lemma is very useful to compute $\D_{\a}(I_{\Delta}^{(t)})$ in this paper.
 
\begin{lem}[\cite{HT2}, Lemma 1.3] \label{degreecomplex} Let $\a\in \ZZ^n$ such that $G_\a \in \Delta$. Then,
$$\F(\D_\a(I_\D^{(t)}))=\{F\in\F(\lk_\D(G_\a))\mid \sum_{i\notin F\cup G_\a}a_i\leqslant t-1\}.$$
\end{lem}

\section{The arboricity and the circumference of a matroid}

Let $M$ be a matroid on the ground set $V=[n]$. The arboricity of $M$, denoted by $a(M)$, is the minimum number of bases needed to cover all elements of the matroid $M$. In this section, we establish a sharp bound for $a(M)$ in terms of the circumstance of its dual matroid $M^*$. 

Define
$$\gamma(M) = \min\{|\S| \mid  \emptyset\ne \S\subseteq \B(M) \text{ and } \bigcap_{B\in \S} B = \emptyset\}.$$ 
For any $\emptyset\ne\S\subseteq \B(M)$, one can see that
$$\bigcap_{B\in \S} B = \emptyset \Longleftrightarrow \bigcup_{B\in \S} (V\setminus B) = V,$$
so that $\gamma(M) = a(M^*)$.

For the proof of the main theorem, some more preparations are needed. 

\begin{lem} \label{H0} Let $M$ be a matroid. Then, $r(M) = |V| - 1$ if and only if $\B(M) = \{V\setminus x\mid x\in V\}$. In particular,  $\gamma(M) = c(M) = |V|$.
\end{lem}
\begin{proof} It is clear that $\B(M) \subseteq \{V\setminus \{x\} \mid x\in V\}$. Assume $V\setminus x \notin \B(M)$. It implies that $x\in V\setminus\{y\}$ for any $y\ne x$. Then, $M$ is a star with a center $x$, a contradiction. Therefore, $\B(M) = \{V\setminus x\mid x\in V\}$, as required.
\end{proof}

\begin{lem} \label{CB} Let $M$ be a matroid on the ground set $V=[n]$. Then, $$\gamma(M) \leqslant c(M).$$
\end{lem}
\begin{proof} We will prove the assertion by induction on $n$.

If $n= 2$ then $r(M) = 1$, since $M$ is not a star. In this case, $\gamma(M) = c(M)$ from Lemma $\ref{H0}$.

If $n > 2$. We have $r(M)=r(M_{-x})$ for any $x\in V$ and $\B(M_{-x})\subseteq\B(M)$, by $M$ is not a star. For each vertex $x\in V$, let $W_x$ be the maximal subset of $V\setminus \{x\}$ such that $W_x\subseteq B$ for all $B\in \B(M_{-x})$ and $M_x=M|(V\setminus (W_x\cup\{x\}))$. Then, $M_x$ is a matroid and not a star. It is obvious that $B\cup W_x\in\B(M_{-x})$ for any $B\in \B(M_x)$. Let $V_x := W_x\cup \{x\}$.  

If $r(M_x) = 0$, i.e. $M_x = \{\emptyset\}$, for some $x\in V$. Therefore, $W_x$ must belong to $\B(M)$ and $\B(M_{-x})=\{W_x\}$. If there exists $y\in (V\setminus \{x\})\setminus W_x$, then there exists $B\in\B(M)$ such that $y\in B, B\setminus\{y\}\subseteq W_x$. Hence, $W_x\ne B\in \B(M_{-x})$, a contradiction. Thus, $W_x=V\setminus \{x\}$ i.e. $r(M) = |V|-1$. By Lemma $\ref{H0}$, $\gamma(M) = c(M)$. 

Now, we may assume that $r(M_x) \geqslant 1$ for any $x\in V$. The rest of our proof will be shown through the following claims.

\smallskip

\textbf{Claim 1:} For any $x, y\in V$,
$$
V_x = V_y \text{ if and only if } y\in V_x.
$$

\smallskip

In order to prove this claim it suffices to show that if $M_{-x}$ is a star with a center $y$ i. e. $y\in W_x$, then $V_x=V_y$.

Since $y\in W_x$, $y\in B$ for any $B\in\B(M_{-x})$. If $B\in\B(M)\setminus\B(M_{-x})$ then $x\in B$. Therefore, for any $B\in\B(M)$, either $x\in B$ or $y\in B$. Hence, we can list 
$$\B(M) =T\cup \{B_{u+1},\ldots, B_{s}, B_{s+1},\ldots, B_m\},$$
where
\begin{itemize}
\item $x,y\in B$ for $B\in T$ and $|T|=u$;
\item $x\in B_i$ and $y\notin B_i$ for $i=u+1,\ldots,s$; and
\item $x\notin B_i$ and $y\in B_i$ for $i=s+1,\ldots,m$.
\end{itemize}

\noindent and $0\leqslant u < s < m$ by $M$ is not a star.  Let $F_i := B_i\setminus \{x\}$ for $i=u+1,\ldots,s$ and $G_i := B_i\setminus \{y\}$ for $i=s+1,\ldots, m$.

Since $M_{-x}$ is a matroid and the above list, we have
$$\B(M_{-x}) = \{B_{s+1}, \ldots, B_m\}.$$
This yields
$$\{F_{u+1}, \ldots, F_s\} \subseteq \{G_{s+1},\ldots, G_m\}.$$
Similarly, one can see that 
$$\B(M_{-y}) = \{B_{u+1}, \ldots, B_{s}\}$$
and 
$$\{G_{s+1},\ldots, G_m\}\subseteq \{F_{u+1}, \ldots, F_s\}.$$
Consequently, $\{G_{s+1},\ldots, G_m\}= \{F_{u+1}, \ldots, F_s\}$. From this, $M_{-y}$ is a star with a center $x$. In particular, $x\in W_y$.

Now, take $t\in W_x\setminus\{y\}$, then $t\in G_i$ for all $i=s+1,\ldots,m$. This gives $t\in F_i$ for all $i=u+1,\ldots,s$. It implies $t\in W_y$. Thus, $V_x\subseteq V_y$. By an argument analogous, we get $V_y\subseteq V_x$. This statement is as required of the claim. 

\smallskip

\textbf{Claim 2:} For any $x\in V$ and $C\in\C(M)$ such that $x\in C$. Then,
$$
V_x \subseteq C. 
$$

\smallskip

Assume the contrary, that $V_x \not \subseteq C$. Put $y\in V_x \setminus C$. By definition, $C$ is also a circuit of $M_{-y}$ and $C\setminus \{x\}$ is also an independent set of $M_{-y}$. Using Claim 1, we have $M_{-y}$ is a star with a center $x$.  Therefore, $C = \{x\} \cup (C\setminus \{x\})$ is also an independent set of $M_{-y}$, which is a contradiction.

Now, we return to prove that $\gamma(M)\leqslant c(M)$. Take $x\in V$. Since $M$ is not a star, there exists a circuit $C\in\C(M)$ such that $x\in C$. Put $p= |V_x|$. Using Claim 2, we have $p \leqslant |C| \le c(M)$. By our assumption, $M' = M_x$  is a matroid, which is not a star, with $r(M')\geqslant 1$.  

Since the cardinality of the ground set of $M'$ is strictly smaller than $n$, using the induction hypothesis, $q=\gamma(M') \leqslant c(M')$. It is easy to see that $c(M')\leqslant c(M)$, then $q \leqslant c(M)$. Let $B'_1,\ldots, B'_q \in \B(M')$ such that
$$\bigcap_{i=1}^q B'_i =\emptyset.$$
Moreover, $V_y = V_x$ for any $y\in V_x$ by Claim 1. Therefore, $M_y = M'$. This yields $B'\cup (V_x\setminus\{y\}) \in \B(M)$ for any $B'\in \B(M')$. 

Write $V_x = \{x_1,\ldots,x_p\}$ and $W_i = V_x \setminus \{x_i\}$ for $i =1,\ldots,p$.  Put $e := \max\{p,q\}$. For each $i=1,\ldots,e$, we shall define a subset $B_i$ of $V$ as follows:

$$B_i =\begin{cases}
W_i \cup B'_i & \text{ if } i\leqslant \min\{p,q\}; \\
W_p \cup B'_i & \text{ if } p \leqslant i \leqslant q \text { (if  $p\leqslant q$)};\\
W_i \cup B'_q & \text{ if } q \leqslant i \leqslant p \text { (if $q\leqslant p$)}.
\end{cases}
$$
Then, it is clear that $B_1,\ldots, B_e\in\B(M)$. Evidently,
$$
\bigcap_{i=1}^p W_i =\emptyset.
$$
From this, we obtain
$$\bigcap_{i=1}^e B_i =\emptyset.$$
Thus, $\gamma(M) \leqslant e \leqslant c(M)$, as required.
\end{proof}

Now we are in a position to prove the main result of this section.

\begin{thm}\label{Arbor} Let $M$ be a matroid. Then, $a(M) \leqslant c(M^*)$.
\end{thm}
\begin{proof} By $M$ is not a star, $M^*$ has the same the ground set $V$ and is not a star. Using Lemma \ref{CB}, $a(M)=\gamma(M^*)\leqslant c(M^*)$ as required. 
\end{proof}

The following example shows that the bound in Theorem $\ref{Arbor}$ is shap.

\begin{exm} Let $M$ be a matroid such that $M^*$ has the largest circuit $C^*$ with $r(C^*) = 1$. Then, $|C^*| = c(M^*)r(C^*)$. Together with \cite[Theorem $1$]{Ed}, it yields $a(M) \geqslant c(M^*)$. Therefore, $a(M) = c(M^*)$ by Theorem $\ref{Arbor}$.
\end{exm}

\begin{rem} The above theorem still holds true when $M$ is a star but $V$ is not in $\B(M)$. Indeed, let $W$ be a maximal subset of $V$ (in relation to inclusion) such that $W\subseteq B$ for all $B\in \B(M)$ and $M'=M|(V\setminus W)$. Since $M'$ is not a star and one can check that $a(M)=a(M')$ and $M^*=(M')^*$, we have  $a(M)\leqslant c(M^*)$.
\end{rem}

\begin{cor}\label{MB} Let $M$ be a matroid on the ground set $V$. Then, $$c(M)(|V|-r(M)) \geqslant |V|.$$
\end{cor}
\begin{proof} It is clear that $M^*$ has the same ground set $V$ because $M$ is not a star.  Combining \cite[Theorem $1$]{Ed} and  Theorem $\ref{Arbor}$ we get
 $$|V| \leqslant a(M^*)r^*(V)\leqslant c(M)r^*(V);$$
where $r^*(V)$ is rank of $V$ on the dual matroid $M^*$. As $r^*(V) = |V| - r(M)$, so that $|V|\leqslant c(M)(|V|-r(M))$, as required.
\end{proof}

We conclude this section with some remarks on the arboricity of graphs. For a simple graph $G$, the arboricity of the matroid $M(G)$ is also called the arboricity of $G$ and also denoted by $a(G)$. Then, $a(G)$ is the minimum number of spanning forests needed to cover all the edges of $G$. Nash-Williams \cite{NW} gave a precise formula for $a(G)$, namely
$$a(G)=\max\left\{\left\lceil\dfrac{e_H}{n_H-1}\right\rceil \mid  H  \text{ is a nontrivial subgraph of } G\right\},$$
where $e_H$ and $n_H$ are the number of edges and vertices of $H$, respectively. 

However, the estimate of this invariant seems to be of independent interest (see e.g. \cite{AMR,DHS,Se}). By this direction, we now set up an upper bound of $a(G)$ by reformulating Theorem $\ref{Arbor}$ for  the graphic matroid of $G$. Let $c^*(G)$ be the size of a largest bond of $G$. From Theorem $\ref{Arbor}$, we obtain.

\begin{cor} Let $G$ be a simple graph. Then, $a(G) \leqslant c^*(G)$.
\end{cor}

Recall that the circumference of $G$, denoted by $c(G)$, is the size of a longest simple cycle of $G$, so that $c(G)$ is the circumference of $M(G)$ as well. Since $M^*(G))$ is not a star whenever $G$ has no leaves, by reformulating  Theorem $\ref{Arbor}$ for $M^*(G)$  we obtain.

\begin{cor} Let $G$ be a simple graph without leaves. Then, the minimal number of spanning forests of $G$ without edges in common is at most $c(G)$.
\end{cor}

\section{The Castelnuovo-Mumford regularity of symbolic powers}

Throughout of this section, let $\D$ be the matroid complex of a matroid $M$ with a positive rank on the ground set $V=[n]$. Let $I$ be the Stanley-Reisner ideal of $\D$ in $S=K[x_1,\ldots,x_n]$ and let $t$ be a positive integer. It is well-known that $I_\D^{(t)}$ is Cohen - Macaulay (see \cite{MT,V}), so that in order to compute $\reg(I_\D^{(t)})$ it suffices to investigate the top local cohomology module $H^d_{\m}(S/I_\D^{(t)})$ where $d={r(M)}$. By virtue of Lemma \ref{T}, it leads  to study the (non)-vanishing of some reduced homology groups of certain degree complexes. 

Firstly, we shall prove some lemmas for degree complexes.

\begin{lem} \label{link} Let $\a\in\NN^n$ such that $\Gamma=\D_\a(I^{(t)})\ne\emptyset$. Assume that $n$ is a vertex of $\Gamma$ and $\lk_\D(n)=\{1,\ldots,p\}$ for some $1\leqslant p\leqslant n-1$. Put $\b=(a_1,\ldots,a_p)\in\NN^p$, and $r=a_{p+1}+\cdots+a_{n-1}$ if $p<n-1$ and $r=0$ if $p=n-1$. Then,
$$\lk_{\Gamma}(n)=\D_{\b}(I_{\lk_\D(n)}^{(t-r)}).$$
\end{lem}
\begin{proof} By $n\in\Gamma$ and Lemma \ref{degreecomplex}, there exists $B\in\B(M)$ such that $n\in B$ and $\sum_{i\notin B}a_i\leqslant t-1$. Then, $r\leqslant \sum_{i\notin B}a_i\leqslant t-1$ by $B\setminus\{n\}\subseteq [p]$. Using Lemma \ref{degreecomplex}, we can list 
$$\B(M) =\{B_1,\ldots, B_u, B_{u+1},\ldots, B_{v}, B_{v+1},\ldots, B_{s}, B_{s+1},\ldots, B_{m}\},$$
where
\begin{itemize}
\item $\F(\Gamma)= \{B_1,\ldots, B_u, B_{u+1},\ldots, B_{v}\}$ for $n\in \bigcap_{i=1}^u B_i$ and $n\notin\bigcup_{i=u+1}^v B_i$;
\item $n\in B_i$ for $i=v+1,\ldots,s$; and
\item $n\notin B_i$ for $i=s+1,\ldots,m$.
\end{itemize}
Therefore, $$\F(\lk_{\Gamma}(n))=\{B_1\setminus\{n\},\ldots, B_u\setminus\{n\}\}$$
and $$\F(\lk_{\D}(n))=\{B_1\setminus\{n\},\ldots, B_u\setminus\{n\}, B_{v+1}\setminus\{n\},\ldots, B_s\setminus\{n\} \}.$$
Using again Lemma \ref{degreecomplex}, for $i=1,\ldots, u, v+1,\ldots, s$, one can check that $B_i\setminus\{n\}\in \F(\D_{\b}(I_{\lk_\D(n)}^{(t-r)}))$ if and only if $i=1, \ldots, u$ (note that the set of vertices of $\lk_\D(n)$ is $[p]$). This implies our assertion.
\end{proof}
 
\begin{lem} \label{restrict} Let $\a\in\NN^n$ such that the degree complex $\Gamma=\D_\a(I^{(t)})\ne\emptyset$. Assume $a_n=\min\{a_i\mid i=1,\ldots, n\}$. Put $\b=(a_1,\ldots,a_{n-1})\in\NN^{n-1}$. Then,
$$\Gamma_{-n}=\D_{\b}(I_{\D_{-n}}^{(t-a_n)}).$$
\end{lem}
\begin{proof} In the same way as in above, we can list 
$$\B(M) =\{B_1,\ldots, B_u, B_{u+1},\ldots, B_{v}, B_{v+1},\ldots, B_{s}, B_{s+1},\ldots, B_{m}\},$$
where
\begin{itemize}
\item $\F(\Gamma)= \{B_1,\ldots, B_u, B_{u+1},\ldots, B_{v}\}$ for $n\in \bigcap_{i=1}^u B_i$ and $n\notin\bigcup_{i=u+1}^v B_i$;
\item $n\in B_i$ for $i=v+1,\ldots,s$; and
\item $n\notin B_i$ for $i=s+1,\ldots,m$.
\end{itemize}
Therefore, $$\F(\Gamma_{-n})=\{B_1\setminus\{n\},\ldots, B_u\setminus\{n\}, B_{u+1},\ldots, B_{v}\}.$$
On the other hand, since $M$ is a matroid and not a star, $M|[n-1]$ is a matroid with the same rank. Thus,
$$\F(\D_{-n})=\{B_{s+1},\ldots, B_m, B_{u+1},\ldots, B_{v}\}.$$
This yields, for each $i=1,\ldots,u$, there exists $p\in\{s+1,\ldots,m,u+1,\ldots,v\}$ such that $B_i\setminus\{n\}=B_p\setminus\{q\}$ for some $q\in B_p$. If $p\in\{s+1,\ldots,m,\}$, then 
\begin{align*}
\sum_{j\in[n],j\notin B_p}a_j=\sum_{j\in[n],j\notin B_p\setminus\{q\}}a_j-a_q=\sum_{j\in[n],j\notin B_i\setminus\{n\}}a_i-a_q\\
=\sum_{j\in[n],j\notin B_i}a_i+a_n-a_q\leqslant \sum_{j\in[n],j\notin B_i}a_i\leqslant t-1
\end{align*}
by Lemma \ref{degreecomplex}, which is a contradiction. Hence, $p\in\{u+1,\ldots,v\}$. It follows 
 $$\F(\Gamma_{-n})=\{B_{u+1},\ldots, B_{v}\}.$$ 
Therefore, our assertion will come from Lemma \ref{degreecomplex}. 
\end{proof}

\begin{rem}
The condition $a_n=\min\{a_i\mid i=1,\ldots, n\}$ in this Lemma can not remove. For instance, we consider an example in which $\Gamma_{-n}$ do not need pure. Let $M$ be a matroid which forms a square i.e. $\B(M)=\{\{1,2\},\{2,3\}, \{3,4\}, \{4,1\}\}$. Let $\a=(1,8,3,2)\in\NN^4$ and $t=11$. By Lemma \ref{degreecomplex}, the degree complex $\Gamma=\D_\a(I^{(t)})$ has the facet set $\{\{1,2\},\{2,3\}, \{3,4\}\}$. From this, $\Gamma_{-3}$ is not pure.
\end{rem}

\begin{lem}\label{circ-link} Let $M$ be a matroid and $x$ is a vertex of $M$. Let $\Delta$ be the simplicial complex of $M$ and let $M'$ be the matroid w.r.t. $\lk_{\D}(x)$. Then, $c(M') \leqslant c(M)$.
\end{lem}
\begin{proof} Observe that 
$$M' = \{F \setminus \{x\} \mid F\in M \text{ and } x\in F\}.$$

Let $C$ be a circuit of $M'$ such that $|C|=c(M')$. Then, for any $v\in C$, $C\setminus\{v\}\in M$, so  $(C\setminus\{v\})\cup \{x\}\in M$. If $C\notin M$, then it is also a circuit of $M$, and then $|C|\leqslant c(M)$. If $C\in M$, then $C\cup \{x\} \notin M$ by definition, and then $C\cup \{x\}$ is a circuit of $M$. Hence, $c(M') < |C\cup\{x\}| \leqslant c(M)$, as required.
\end{proof}

\begin{lem} \label{upper} Let $\a\in\NN^n$ such that $\Gamma=\D_\a(I^{(t)})$ is not acyclic.
Then, $$|\a|=\sum_{i=1}^n a_i \leqslant c(M)(t-1).$$
\end{lem}
\begin{proof} We now proceed by induction on $t+n\geqslant 3$. If $t + n = 3$, then $t=1$ and $n=2$. In particular, $\B(M)=\{1,2\}$ and $c(M)=2$. In this case, one can see that $\Delta_{\a}(I) =\D$. Lemma \ref{degreecomplex} yields
$$a_1 \leqslant 0 \text{ and } a_2 \leqslant 0,$$
i. e. $a_1=a_2 = 0$. Therefore, $|\a| = 0 = c(M)(t-1)$, as required.

Assume $t + n \geqslant 4$ and $a_n =\min\{a_i\mid i=1,\ldots, n\}$. If $a_n \geqslant 1$, put $b_i=a_i-1$ for all $i$ and $\b=(b_1,\ldots,b_n)\in \NN^n$. Then, for any $B\in \B(M)$, we have
$$\sum_{i\notin B} b_i = \sum_{i\notin B} a_i -(n-r(M)).$$
Applying Lemma \ref{degreecomplex} and $\Gamma\ne\emptyset$, $$\Delta_{\b}(I^{(t-(n-r(M)))}) = \Delta_{\a}(I^{(t)}).$$
By $(t-(n-r(M)))+ n < t+n$ and the induction hypothesis, we obtain
$$|\b| \leqslant c(M)(t-(n-r(M))-1).$$
Therefore, by Corollary $\ref{MB}$, $$|\a| =|\b|+n  \leqslant c(M)(t-1) + n - c(M)(n-r(M))\leqslant c(M)(t-1).$$ 

Assume $a_n = 0$. By \cite[Theorem 1.6]{MT} we have $\Gamma$ is not acyclic and is Cohen–Macaulay of dimension $d=r(M)-1$, hence $\h_d(\Gamma;K)\ne 0$. By \cite[Lemma 2.1]{H}, we have the following exact sequence:
$$\h_{d-1}(\lk_\Gamma(n); K) \longrightarrow \h_d(\Gamma;K) \longrightarrow \h_id(\Gamma_{-n}; K).$$
Since $\h_d(\Gamma;K)\ne 0$, either $\h_{d-1}(\lk_\Gamma(n); K)\ne 0$ or $\h_d(\Gamma_{-n}; K)\ne 0$.

\medskip

We next distinguish two cases:

\medskip

\noindent {\it Case $1$}: $\h_{d-1}(\lk_\Gamma(n); K)\ne 0$. Applying Lemma \ref{star}, the simplicial complex $\lk_\Gamma(n)$ is matroid of dimension $d-1$ which is not a cone. Then, it is also not acyclic by Lemma \ref{acyclic}. One may assume the set of vertices of $\lk_{\D}(n)$ is $\{1,\ldots,p\}$ for some $1\leqslant p\leqslant n-1$. Let $\b = (a_1,\ldots,a_p)\in \NN^p$, and $r = a_{p+1}+\cdots+a_{n-1}$ if $p<n-1$ and $r=0$ if $p=n-1$. Then, by Lemma \ref{link},
$$\lk_{\Gamma}(n) = \Delta_{\b}(I_{\lk_{\D}(n)}^{(t-r)}).$$
It can see that $(t-r)+p < t + n$. Applying again Lemma \ref{star} and Lemma \ref{acyclic}, $\lk_\D(n)$ is the non-acyclic matroid complex of a matroid $M'$. By the induction hypothesis and Lemma $\ref{circ-link}$ we have $|\b| \leqslant c(M')(t-r-1)\leqslant c(M)(t-r-1)$. Thus,
\begin{align*}
|\a| = |\b|+(a_{p+1}+\cdots+a_n) = |\b|+r +a_n = |\b| +r\\
\leqslant c(M)(t-r-1)+r \leqslant c(M)(t-1). 
\end{align*}

\noindent {\it Case $2$}: $\h_{d}(\Gamma_{-n}; K)\ne 0$. Then, by Lemma \ref{acyclic}, $\Gamma_{-n}$ is matroid of dimension $d$ which is not a cone on the ground set $\{1,\ldots,n-1\}$. Let $\b = (a_1,\ldots,a_{n-1})\in \NN^{n-1}$. Then, by Lemma \ref{restrict},
$$\Gamma_{-n}= \Delta_{\b}(I_{\D_{-n}}^{(t-a_n)}) = \Delta_{\b}(I_{\D_{-n}}^{(t)}) .$$
If $\D_{-n}$ is a cone, then so is $\Gamma_{-n}$ by Lemma \ref{degreecomplex}, a contradiction. Hence, $\Gamma_{-n}$ is a matroid complex which is not a cone on the ground set $\{1,\ldots,n-1\}$. By the induction hypothesis, $|\b| \leqslant c(M_{-n})(t-1)$. Since $\C(M_{-n}) \subseteq \C(M)$, we have $c(M_{-n}) \leqslant c(M)$, so
$$|\a| = |\b| + a_n = |\b| \leqslant c(M)(t-1).$$
Thus, $|\a| \leqslant c(M)(t-1)$ in both cases, and the proof  is complete.
\end{proof}

We are now in a position to prove the main result of this paper.

\begin{thm}\label{regsym} Let $\D$ be the simplicial complex of a matroid $M$.  Then, 
$$\reg(I_{\Delta}^{(t)}) = c(M)(t-1)+r(\core(M))+1, \ \text{ for all } t\geqslant 1.$$
\end{thm}
\begin{proof} Without loss of generality, we may assume that $M=\core(M)$ so that $\D$ is not a cone. Let $I = I_{\Delta}$ and $d = r(M)$ so that $d = \dim S/I^{(t)}$. From \cite[Theorem 3.5]{MT}, we have $I^{(t)}$ is Cohen-Macaulay, hence $a_{i}(S/I^{(t)})=-\infty$ for all $i < d$. Since $\reg(I^{(t)}) = \reg S/I^{(t)} +1$, it remains to prove that $a_{d}(S/I^{(t)}) = c(M)(t-1)$. 

Let $\a$ be a vector in $\ZZ^n$  such that $H_\m^d(S/I^{(t)})_\a\ne 0$. Because $\Delta(I^{(t)}) = \Delta$, by Lemma \ref{T} we have $\h_{d-|G_\a|-1}(\D_\a(I^{(t)});K)\ne 0$ and $G_\a \in\Delta$. Let $\Gamma=\lk_{\D}(G_\a)$. By Lemma \ref{star},  $\Gamma$ is a matroid complex, which is not a cone, of dimension $(d-1)-|G_\a|$ .

Define $\b \in \NN^n$  by
$$
b_i=
\begin{cases}
a_i & \text{ if } a_i \geqslant 0,\\
0 & \text{ otherwise.}
\end{cases}
$$
Then, by applying Lemma $\ref{degreecomplex}$ we conclude that $\D_\a(I^{(t)})= \D_\b(I_{\Gamma}^{(t)})$. In particular, $ \D_\b(I_{\Gamma}^{(t)})$ is not acyclic. Let $M'$ be the matroid w.r.t. $\Gamma$. Then, $c(M')\leqslant c(M)$ according to Lemma $\ref{circ-link}$.  On the other hand, $|\b|\leqslant c(M')(t-1)$ by Lemma \ref{upper}. Therefore, $|a| \leqslant |b| \leqslant c(M)(t-1)$. It follows that $a_d(R/I^{t}) \leqslant c(M)(t-1)$. 

For the reverse inequality $a_{d}(S/I^{(t)}) \geqslant c(M)(t-1)$, it suffices to prove that there is  a vector $\a\in\NN^n$ such that $\h_{d-1}(\D_\a(I^{(t)});K)\ne 0$ with $|\a|\geqslant c(M)(t-1)$. We will prove the existence of such a vector by induction on $n$. 

Observe that if $c(M)=n$, then $\B(M)=\{[n]\setminus \{i\}\mid  i\in [n]\}$ by Lemma $\ref{H0}$. Let $\a = (t-1,\ldots,t-1)\in \NN^n$. By Lemma $\ref{degreecomplex}$ we have $\D_\a(I^{(t)}) = \D$, so $\a$ is a desirable vector by Lemma $\ref{acyclic}$.

We return to prove the existence of $\a$ by induction on $n\geqslant 2$. If $n=2$, then $c(M) = 2$ because $M$ is not a star, and then $c(M) = n$. This case has done by the above observation.

If $n>2$, let $C$ be a circuit of $M$ such that $|C|=c(M)$. By the above observation we may assume that $|C|<n$, so we can assume that $C=\{1,\ldots,c\}$ for $c<n$. Let $M'$ be the matroid w.r.t. $\D_{-n}$. Note that $\C(M')\subseteq C(M)$ and $C$ is also a circuit of $M'$, so  $c(M')=c(M)$. By the induction hypothesis, there is a vector $\b\in\NN^{n-1}$ such that $|\b| \geqslant c(M')(t-1)$ and $\h_{d-1}(\D_\b(I_{\D_{-n}}^{(t)});K)\ne 0$. Put $\a=(b_1,\ldots,b_{n-1},0)\in\NN^n$ and $\Gamma = \D_\a(I^{(t)})$. By Lemma $\ref{degreecomplex}$ we see that $\D_\b(I_{\D_{-n}}^{(t)})$ is a subcomplex of $\Gamma$, so $\Gamma\ne\emptyset$.  By Lemma \ref{restrict}, we have $\Gamma_{-n}=\D_{\b}(I_{\D_{-n}}^{(t-a_n)})=\D_{\b}(I_{\D_{-n}}^{(t)})$, hence $\h_{d-1}(\Gamma_{-n}; K)\ne 0$. Note that $\dim(\lk_{\Gamma}(n)) =  d-2$, so $\h_{d-1}(\lk_{\Gamma}(n)); K)= 0$. By \cite[Lemma 2.1]{H}, we have the following exact sequence
$$\h_{d-1}(\Gamma;K) \longrightarrow \h_{d-1}(\Gamma_{-n}; K)\longrightarrow \h_{d-1}(\lk_\Gamma(n); K).$$
Because $\h_{d-1}(\lk_{\Gamma}(n)); K)= 0$ and $\h_{d-1}(\Gamma_{-n}; K)\ne 0$,  it yields $\h_{d-1}(\Gamma; K)\ne 0$. Note that $c(M')=c(M)$, so $|a| = |b| \geqslant c(M)(t-1)$, and the proof is complete.
\end{proof}

As a corollary, we can obtain a characterization of a matroid whose the $t$-th symbolic power of a Stanley-Reisner ideal has linear resolution.

\begin{thm}\label{regsym-cor} Let $\D$ be the matroid complex of a matroid $M$ on the ground set $V=[n]$. Let $I$ be the Stanley-Reisner ideal of $\D$. Then, the following conditions are equivalent:
\begin{enumerate}
\item[(i)] $I^{(t)}$ has linear resolution for all $t\geqslant 1$,
\item[(ii)] $I^{(t)}$ has linear resolution for some $t\geqslant 1$,
\item[(iii)] $M$ forms as $U_{k,n}$ for some $1\leqslant k<n$.
\end{enumerate}
\end{thm}
\begin{proof}(i) $\Rightarrow$ (ii) is clear. (iii) $\Rightarrow$ (i) holds true by Theorem \ref{regsym} with $c(M)=k+1$ and $r(\core(M))=k$ in this case. (ii) $\Rightarrow$ (iii): Take an arbitrary circuit $C\in\C(M)$ and put $f_C=(\prod_{i\in C}x_i)^t$. It is clear that $f_C\in I^{(t)}$ and $\dfrac{f_C}{x}\not\in I^{(t)}$ for any variable $x$ which is a divisor of $f_C$. This yields $f_C$ is a minimal monomial generator of $I^{(t)}$. By our assumption, every circuits of $M$ has the same size $k$ for $2\leqslant k\leqslant n$. By Theorem \ref{regsym} and $M$ is not star,  $r(\core(M))=r(M)=k-1$. Therefore, every subsets of $[n]$ which has $(k-1)$ elements must belong to $M$. It implies that $M$ is of the form $U_{k-1,n}$, and the theorem follows.
\end{proof}

\subsection*{Acknowledgment} This paper was done while the first author was visiting Vietnam Institute for Advanced Study in Mathematics (VIASM).  He would like to thank the VIASM for hospitality and financial support and also thanks the support by Vietnam National Foundation for Science and Technology Development (NAFOSTED) under grant number 101.04-2016.21. %The second author partially supported by the National Foundation for Science and Technology Development (Vietnam) under grant number 101.04-2015.02. 

\end{document}